\documentclass[]{amsart}

\usepackage{MnSymbol}
\usepackage[cmtip,arrow]{xy}  
\usepackage{pb-diagram,pb-xy}  
\usepackage{graphicx}    
\usepackage{amsfonts}
\usepackage[latin1]{inputenc}
\usepackage{caption}
\usepackage{subcaption}
\usepackage{comment}
\usepackage{enumerate}
\usepackage{amsmath}
\usepackage{bbm}

\usepackage{tikz-cd}
\usepackage{tikz}
\usetikzlibrary{matrix}
\usetikzlibrary{positioning, calc}
\usetikzlibrary{shapes, fit}
\usetikzlibrary{arrows}  
\usetikzlibrary{calc,3d}
\usetikzlibrary{decorations,decorations.pathmorphing}
\usetikzlibrary{through}
\tikzset{ext/.style={circle, draw,inner sep=1pt},int/.style={circle,draw,fill,inner sep=1pt},nil/.style={inner sep=1pt}}
\tikzset{exte/.style={circle, draw,inner sep=1pt},inte/.style={circle,draw,fill,inner sep=3pt}}
\tikzset{diagram/.style={matrix of math nodes, row sep=1.5em, column sep=0.5em, text height=1.5ex, text depth=0.25ex}}
\tikzset{diagram2/.style={matrix of math nodes, row sep=0.5em, column sep=0.5em, text height=1.5ex, text depth=0.25ex}}

\newcommand{\aor}{\rcirclearrowleft}
\newcommand{\aol}{\lcirclearrowright}

\newcommand{\cP}{{\mathcal{P}}}
\newcommand{\cQ}{{\mathcal{Q}}}
\newcommand{\cM}{{{M}}}
\newcommand{\Br}{{\mathsf{Br}}}
\newcommand{\FM}{{\mathsf{FM}}}
\newcommand{\e}{{\mathsf{e}}}
\newcommand{\bbo}{{\mathbbm{1}}}

\newcommand{\K}{{\mathbb{K}}}
\newcommand{\R}{{\mathbb{R}}}
\newcommand{\bbS}{{\mathbb{S}}}
\newcommand{\mF}{{\mathcal{F}}}

\DeclareMathOperator{\im}{Im}
\DeclareMathOperator{\id}{id}
\DeclareMathOperator{\Free}{Free}
\DeclareMathOperator{\End}{End}
\DeclareMathOperator{\Hom}{Hom}

\newtheorem{theorem}{Theorem}[]
\newtheorem{corollary}[theorem]{Corollary}

\newtheorem{lemma}[theorem]{Lemma}
\newtheorem{proposition}[theorem]{Proposition}

\newtheorem{prop-def}[theorem]{Proposition-definition}
\newtheorem{remark}[theorem]{Remark}
 \newtheorem{defi}{Definition}
 
\begin{document}

\title{Operadic Torsors}

\author{Ricardo Campos}
\address{Institute of Mathematics\\ University of Zurich\\ Winterthurerstrasse 190 \\ 8057 Zurich, Switzerland}
\email{ricardo.campos@math.uzh.ch}

\author{Thomas Willwacher}
\address{Institute of Mathematics\\ University of Zurich\\ Winterthurerstrasse 190 \\ 8057 Zurich, Switzerland}
\email{thomas.willwacher@math.uzh.ch}

\thanks{This work was partially supported by the Swiss National Science Foundation, grant 200021\_150012 and the NCCR SwissMAP funded by the Swiss National Science Foundation.}

\begin{abstract}
We introduce the notion of operadic torsors and operadic quasi-torsors.
We show that if an operadic (quasi-)torsor between two operads exists, then these operads are (quasi-)isomorphic.
As an application we present the (arguably) shortest known proof of the Deligne conjecture.
\end{abstract}

\maketitle

\section{Introduction}
Suppose $G$ and $H$ are groups. Then a $G$-$H$ torsor is a set $M$ with a free and transitive action of $G$ from the left and a compatible free and transitive action of $H$ from the right.
It is an elementary exercise to check that if a $G$-$H$ torsor exists, then the groups $G$ and $H$ are isomorphic.
In fact, showing the existence of a torsor is a neat trick to show the isomorphism of two groups without actually constructing an isomorphism.

The purpose of this paper is to show that the above notions, and the trick, carry over to the operadic setting.
To this end we make the following definition.
\begin{defi}\label{def:optorsor}
Let $\mathcal P$ and $\mathcal Q$ be two differential graded operads and let $M$ be a $\mathcal{P}-\mathcal Q$ operadic differential graded bimodule, i.e., there are compatible actions $$\mathcal P \aol M \aor \mathcal Q.$$
We say that $\cM$ is a $\mathcal P$-$\mathcal Q$ torsor if there is an element $\mathbf 1 \in M^0(1)$ such that the canonical maps
\begin{equation}\label{equ:modmaps}
\begin{aligned}
 \cP &\to \cM  \quad\quad\quad\quad\quad & \cQ & \to \cM \\
 p   &\mapsto p\circ (\mathbf 1,\dots, \mathbf 1)  & q &\mapsto \mathbf 1\circ q \\
\end{aligned}
\end{equation}
are isomorphisms. We similarly say that $\cM$ is a quasi-torsor if the maps \eqref{equ:modmaps} are quasi-isomorphisms.
\end{defi}

The above definition connects to the notion of torsor over groups as follows. Suppose $G$ and $H$ are groups and $M$ is a $G$-$H$-torsor.
Then we may consider the group algebras $\K[G]$, $\K[H]$ as operads with only unary operations, and the bimodule $\K[M]$ is an operadic torsor in the sense of our definition.

It is an elementary exercise to check that if an operadic $\mathcal P$-$\mathcal Q$-torsor exists, then the operads $\cP$ and $\cQ$ are isomorphic.
The main result of this paper is the less elementary assertion that the analogous result for quasi-torsors also holds.

\begin{theorem}\label{thm:main}
Let $\cP$ and $\cQ$ be differential graded operads. Then an operadic $\cP$-$\cQ$-quasi-torsor exists if and only if $\cP$ and $\cQ$ are quasi-isomorphic operads.
\end{theorem}

As an application one can give a very short and natural proof of the Deligne conjecture: One merely has to note that the (chains of the) configuration space of points in the upper half-plane form an operadic torsor between the braces operad and the configuration space of points in $\R^2$. We will make this application explicit in section \ref{sec:deligne}.

Theorem \ref{thm:main} settles the question of when an operadic quasi-torsor exists.
Our second main result shows that an operadic quasi-torsor is essentially unique, if it exists.
\begin{theorem}\label{thm:uniqueness}
 Let $\cP$ and $\cQ$ be differential graded operads and let $\cM$ be an operadic $\cP$-$\cQ$ quasi-torsor. 
 Then there is a zig-zag of quasi-isomorphisms of operads and bimodules connecting the actions
  $$\mathcal P \aol \cM \aor \mathcal Q$$
  to the canonical operadic torsor
  $$\mathcal Q \aol \cQ \aor \mathcal Q.$$
\end{theorem}

The result can be used to significantly shorten the proof of the homotopy braces formality theorem \cite{brformality}, see section \ref{sec:brformality}.

\section{Notation and prerequisites}
We work over a field $\K$ of characteristic zero, i.e., all vector spaces or differential graded (or dg for short) vector spaces are $\K$-vector spaces. 
By default, all vector spaces, operads, modules or bimodules are differential graded, even when not explicitly stated.
We work in homological conventions, i.e., the differentials have degree $-1$.
For $V$ a (dg) vector space, we denote by $V^d$ the subspace of elements of degree $d$. $V[r]$ denotes the degree shifted vector space defined such that $(V[r])^d = V^{d+r}$
We denote the degree of a homogeneous element $v\in V$ by $|v|$.

A good introduction to the theory of operads can be found in the textbook \cite{lodayval}, whose conventions we will mostly follow.
In particular, an $\bbS$-module $\cP$ is a collection of (dg) vector spaces $\cP(N)$, $N=0,1,2,\dots$, with right actions of the symmetric groups $\bbS_N$.
The category of $\bbS$-modules comes with a monoidal product $\circ$, see \cite[section 5.1.6]{lodayval}, and operads are monoids with respect to this product. Concretely, an operad $\cP$ is defined by providing morphisms
\begin{align*}
 \cP\circ \cP &\to \cP & 1&\to \cP  
\end{align*}
satisfying natural compatibility relations. An operad is augmented if it is equipped with an additional morphism $\cP\to 1$, right inverse to the unit map above.
For those operads one may define the operadic bar construction, and dually for coaugmented cooperads the operadic cobar construction as in \cite[section 6.5]{lodayval}. We denote the kernel of the augmentation by $\bar\cP$.

Operadic left-, right- or bimodules over $\cP$ are left- or right- or bimodules of the monoid $\cP$. For example, a left $\cP$-module (also called $\cP$-algebra) is an $\bbS$-module $\cM$ together with a map of $\bbS$-modules
\[
 \cP \circ \cM \to \cM
\]
satisfying obvious compatibility relations with the operadic composition and unit. Let us point out that the notion of operadic bimodule has been introduced in \cite{KaMa}, while that of operadic right modules goes back to \cite{smirnov}.

If $\cM$ is a right $\cQ$-module, we denote the action of $c_1,\dots, c_n$ on $m\in \cM(n)$ by $m(c_1,\dots, c_n)$.
If $f$ is a map of complexes, the respective homology map will be denoted by $[f]$.
We will denote the unit of an operad $\cP$ by $1_\cP\in \cP(1)$.

The homotopy theory of operads and algebras over operads has been developed in \cite{hinich, GJ}.
A somewhat more general treatment including the cases of operadic right modules and bimodules is contained in \cite{Frbook}.
In the present work we will sometimes allude to the general model categorial result shown in the aforementioned references.
However, to keep the exposition elementary, and since we do not use any deep statement of loc. cit., we will make explicit all arguments of model categorial nature below.

\section*{Acknowledgements}
We thank B. Fresse for helpful discussions.

\section{The proof of Theorems \ref{thm:main} and \ref{thm:uniqueness}} \label{sec:fullproof}

As a first reduction, we will replace the torsor $M$ by a suitable resolution.

\begin{proposition}\label{prop:resolution}
Let $M$ be an operadic $\cP$-$\cQ$-bimodule. Then there is a $\cP$-$\cQ$ bimodule $M_\infty$, quasi-free as right $\cQ$-module with a quasi-isomorphism $M_\infty\to M$. If $M$ is an operadic quasi-torsor then $M_\infty$ is a quasi-torsor. Furthermore, one can find a right inverse $\mu:M_\infty \to \cQ$ of the quasi-isomorphism of right $\cQ$-modules $\cQ\to M_\infty$ as in \eqref{equ:modmaps}.
\end{proposition}

The proof of the proposition will be given in section \ref{sec:resolution}. We will believe it for now and use it to show Theorems \ref{thm:main} and \ref{thm:uniqueness}. Indeed, Proposition \ref{prop:resolution} clearly allows us to replace $\cP\aol M\aor \cQ$ by its resolution 
$$
 \begin{tikzcd}[column sep=0.5em]				
 \cP \arrow{d}{\id} &\aol & M_\infty \arrow{d}{\sim} & \aor & \cQ \arrow{d}{\id} \\
 \cP & \aol  & M &\aor & \cQ 
 \end{tikzcd}
$$
and hence reduces the statements of Theorems \ref{thm:main} and \ref{thm:uniqueness} to the following result.

\begin{theorem}\label{main}
Let $\mathcal P$ and $\mathcal Q$ be two dg operads and let $M$ be a $\mathcal{P}-\mathcal Q$ operadic quasi-torsor.
Suppose furthermore that there is a map $\mu \colon    M \to \cQ$ of right $\cQ$-modules such that $\mu \circ q = \id_{\cQ}$.
Then the operads $\cP$ and $\cQ$, and the bimodules $\mathcal P \aol M \aor \mathcal Q$, $\mathcal P \aol \cP \aor \mathcal P$ and $\mathcal Q \aol \cQ \aor \mathcal Q$
are quasi-isomorphic.
\end{theorem}

For any dg $\mathbb S$-module $N$ with differential $d_N$ one defines its endomorphism operad $\End N (n) = \{\varphi\colon    N^{\otimes n} \to N\}$ with the composition of maps as the operadic composition and differential 

\begin{align*}
\partial \varphi(x_1,\dots,x_n) &= d_N\varphi(x_1,\dots,x_n) -(-1)^{|\varphi|}\varphi(d_{N^{\otimes n}}(x_1,\dots,x_n))\\
&=d_N\varphi(x_1,\dots,x_n) \pm \varphi(d_Nx_1,\dots,x_n) \pm  \dots \pm \varphi(x_1,\dots,d_Nx_n).\footnotemark
\end{align*}\footnotetext{The signs are determined by the usual Koszul sign rules.}

Given two dg $\mathbb S$-modules $M$ and $N$ and a pair $(f,g)$, where $f\colon    N\to M$ and $g\colon    M\to N$ are maps of dg $\mathbb S$-modules, one can construct maps of $\mathbb S$-modules
\begin{equation}\label{end map}
\begin{aligned}
 \overline{f}\colon    & \End N &\to &\End M\\ 
              & \lambda\colon   N^{\otimes k}\to N &\mapsto & \left[ m_1,\dots, m_k \mapsto f\circ\lambda(g(m_1),\dots,g(m_k)) \right]
\end{aligned}
\end{equation}
and $\overline{g}\colon     \End M \to \End N$ in the same way. 
To check that $\overline{f}$ commutes with the differentials, let us consider arbitrary $m_1,\dots,m_n \in M^{\otimes n}$.

\begin{flalign*}
&\partial\overline{f}(\lambda)(m_1,\dots,m_n)= \\
&=d_M\overline{f}(\lambda)(m_1,\dots,m_n) \pm \overline{f}(\lambda)(d_Mm_1,\dots,m_n) \pm \dots \pm \overline{f}(\lambda)(m_1,\dots,d_Mm_n)\\
&=d_M f\circ \lambda (g(m_1),\dots,g(m_n)) \pm f\circ \lambda(g(d_Mm_1),\dots,g(m_n)) \pm \dots \pm f\circ\lambda(g(m_1),\dots,g(d_Mm_n))\\
&=f\left( d_M \lambda (g(m_1),\dots,g(m_n)) \pm  \lambda(d_Mg(m_1),\dots,g(m_n)) \pm \dots \pm \lambda(g(m_1),\dots,d_Mg(m_n))\right)\\
&=f\circ\partial\lambda(g(m_1),\dots,g(m_n))\\
&=\overline{f}(\partial\lambda)(m_1,\dots,m_n).
\end{flalign*}

Therefore $\overline{f}$ and $\overline{g}$ are well defined maps of dg $\mathbb S$-modules.

\begin{remark}\label{iso=>iso}
If $f$ and $g$ are isomorphisms (not necessarily inverse to each other),  the injectivity and surjectivity of $\overline{f}$ are easily checked, so $\overline{f}$ and $\overline{g}$ (by symmetry) are isomorphisms.
\end{remark}

We start the proof of Theorem \ref{main} with a natural relation between the endomorphism operads of two bimodules.

\begin{lemma}
Let $M$ and $N$ be dg $\mathbb S$-modules and $f\colon   N\to M$ and $g\colon   M\to N$ are quasi-isomorphims of dg $\mathbb S$-modules.
Then $\overline{f}\colon   \End N \to \End M$, as defined above, is also a quasi-isomorphism.
\end{lemma}
\begin{proof}

Fixed $n\in \mathbb N_0$ and a vector space $V$, the three functors from $Vect$ to $Vect$ given by 
$$W\mapsto W^{\otimes n}, W\mapsto \Hom(W,V) \text{ and } W\mapsto \Hom(V,W)$$
are exact. Therefore there are canonical isomorphisms 
\begin{multline*}
 H(\End N)(n) = H(\Hom(N^{\otimes n}, N)) = \Hom(H(N^{\otimes n}), H(N)) \\=H(\Hom(H(N)^{\otimes n}, H(N)) = \End H(N) (n).
\end{multline*}

This identification is given by the map
\begin{align*}
I\colon   H(\End N) &\to \End H(N) \\
\lambda \colon    N^{\otimes k}\to N &\mapsto 
\left(
\begin{aligned}
I(\lambda)\colon    H(N)^{\otimes k} &\to H(N)\\
     ([n_1],\dots,[n_k])&\mapsto [\lambda(n_1,\dots,n_k)]
\end{aligned}
\right).
\end{align*}

Given $(f,g)$ satisfying the hypothesis of the lemma, we get a second pair $([f],[g])$ given by the induced maps on homology, thus defining $\overline{\left[f\right]}\colon   \End H(N)\to\End H(M)$.

It is easy to check that the following diagram is commutative:

$$\begin{tikzcd}
\End H(N) \arrow{r}{\overline{\left[f\right]}}  & \End H(M)  \\
H(\End N) \arrow{u}{I}\arrow{r}{\left[\overline{f}\right]} & H(\End M) \arrow[swap]{u}{I}
\end{tikzcd}
$$

Since by hypothesis the maps $\left[f\right]$ and $\left[g\right]$ are isomorphisms, $\overline{\left[f\right]}$ is an isomorphism by \ref{iso=>iso}. We conclude that $\left[\overline{f}\right]$ must be an isomorphism as well.
\end{proof}

If $N$ is a right $\cQ$-module one defines its operad of $\cQ$-invariant endomorphisms, $\End_{\cQ}N$ (cf. \cite[section 9.4]{Frbook}), as the subset of its endomorphism operad formed by the morphisms $\lambda \colon    N^{\otimes n}(k)\to N(k)$ of $\mathbb S$-modules such that $\forall x_1\in N(i_1),\dots, x_n \in N(i_n) \text{ and } c_1,\dots,c_k\in \cQ$ $$\lambda(x_1(c_1,\dots,c_{i_1}),\dots,x_n(c_{i_1+...+i_{n-1}+1},...,c_{k})) = \lambda(x_1,\dots, x_n) (c_1,\dots c_k).$$

It is clear that $\End_{\cQ} N$ is closed under operadic composition.  Using the fact that the $N$ is a dg right module over $\cQ$ it is a straightforward calculation to check that the differential restricts to $\End_{\cQ} N$ making $\End_{\cQ} N$ a dg operad.  \\

Given two dg right $\cQ$-modules $M$ and $N$, $f\colon   N\to M$ and $g\colon   M\to N$, maps of $\cQ$-modules, we consider the morphism $\overline{f}\colon   \End N \to \End N$ as defined in \eqref{end map}. If $\lambda\in \End_{\cQ} N$, it is clear that since $f$, $\lambda$ and $g$ commute with the right $\cQ$ action, then $\overline f (\lambda)$ will also commute with the $\cQ$ action, so $\overline{f}$ restricts to a map $\End_{\cQ} N \to \End_{\cQ} M$ that we still denote by $\overline{f}$.

\begin{lemma}
Let $M$ and $N$ be two dg right $\cQ$-modules and let $f\colon   N\to M$ and $g\colon   M\to N$ be two quasi-isomorphisms of right $\cQ$-modules, quasi-inverse to each other. Then the map of $\mathbb S$-modules
$\overline{f}\colon    \End_{\cQ} N \to \End_{\cQ} M$ is also a quasi-isomorphism.
\end{lemma}
\begin{proof}

Let us define 
\begin{align*}
I'\colon   H(\End_{\cQ} N) &\to \End_{H(\cQ)} H(N)\\
\lambda \colon    N^{\otimes k}\to N &\mapsto 
\left(
\begin{aligned}
I'(\lambda)\colon    H(N)^{\otimes k} &\to H(N)\\
  ([n_1],\dots,[n_k])&\mapsto [\lambda(n_1,\dots,n_k)] 
\end{aligned}
\right).
\end{align*} 


Let $i\colon    \End_{\cQ} N \to \End N$ and $j\colon    \End_{H(\cQ)} H(N) \to \End H(N)$ be the inclusion maps.
We have the following commutative diagram that essentially just expresses that $I'$ can be seen as a restriction of $I$:  

$$\begin{tikzcd}
H(\End N) \arrow{r}{I} & \End H(N)\\
H(\End_{\cQ} N) \arrow[hookrightarrow]{u}{[i]} \arrow{r}{I'} & \End_{H(\cQ)} H(N) \arrow[hookrightarrow]{u}[swap]{j}
\end{tikzcd}$$
and therefore $I'$ is injective. Let us define $X(N)$ to be the image of $I'$ inside $\End_{H(\cQ)} H(N)$.

Consider $I'(\lambda)\in X(N)$, where $\lambda \in \End_{\cQ} N$, such that $\partial \lambda=0$.
Then $\overline{f}(\lambda) \in \End_{\cQ} M$ and  $\partial(\overline{f}(\lambda))=0$, so the easy to check equality $\overline{[f]}(I'(\lambda)) = I'(\overline{f}(\lambda))$ tells us that $\overline{[f]} (X(N)) \subset X(M)$.
We obtain then the following diagram
$$
\begin{tikzcd}
H(\End_{\cQ} N) \arrow{d}{\left[\overline{f}\right]} \arrow{r}{I'} &X(N) \arrow{d}{\overline{[f]}} \arrow[hookrightarrow]{r} &\End_{H(\cQ)} H(N) \arrow{d}{\overline{[f]}}\\
H(\End_{\cQ} M) \arrow{r}{I'} &X(M)\arrow[hookrightarrow]{r} &\End_{H(\cQ)} H(M)
\end{tikzcd}
$$

By the symmetry of the problem, the map $\overline{[g]}\colon \End_{H(\cQ)} H(M)\to \End_{H(\cQ)} H(N)$ also restricts to a map $X(M)\to X(N)$ that is clearly an inverse of $\overline{[f]}\colon X(N)\to X(M)$, hence they are both isomorphisms, therefore $\left[\overline{f}\right]$ is an isomorphism as wanted.
\end{proof}

\begin{corollary}\label{qi of End operads}
If $f\colon   N\to M$ is a quasi-isomorphism of right $\cQ$-modules and $g\colon   M\to M$ is a dg map of right $\cQ$-modules such that $g\circ f = \id_N$, then $\overline f \colon    \End_{\cQ} N \to \End_{\cQ}M$ is a quasi-isomorphism of operads.
\end{corollary}
\begin{proof}
It is clear that $g$ is a quasi-inverse of $f$, therefore it is enough to check that $\overline{f}$ commutes with the operadic composition.

Let $c\in \End_{\cQ} N (k)$ and $c_1,\dots,c_k\in \End_{\cQ} N$. For all $m_1,\dots,m_n$ we have

\begin{align*}
\overline{f}(c)(\overline{f}(c_1),\dots, \overline{f}(c_k))(m_1,\dots,m_n) &= \overline{f}(c)(\overline{f}(c_1)(m_1,\dots),\dots, \overline{f}(c_k)(\dots,m_n))\\
&= \overline{f}(c)(f\circ c_1 (g(m_1),\dots),\dots, f\circ c_k(\dots,g(m_n)))\\
&=f\circ c(g\circ f \circ c_1(g(m_1),\dots),\dots,g\circ f\circ c_k(\dots,g(m_n)))\\
&=f\circ c(c_1,\dots,c_k)(g(m_1)\dots,g(m_n))\\
&=\overline{f}(c(c_1,\dots,c_k))(m_1,\dots,m_n).
\end{align*}
\end{proof}

Using the above results we can now show Theorem \ref{main}.

\begin{proof}[Proof of Theorem \ref{main}]

Notice that we can identify the dg operad $\End_{\cQ} \cQ$ with $\cQ$ via $\lambda\in\End_{\cQ} \cQ (k)\mapsto \lambda(1_\cQ,\dots,1_\cQ)$.
Using this identification, under the conditions of Theorem \ref{main} we have the following diagram:

\begin{eqnarray}\label{diagram}
\begin{tikzcd}
 \    &\End_{\cQ} M \arrow{d}{\iota_{\mathbf{1}}} & \ \\
\cP \arrow[swap]{r}{p}\arrow{ur}{p'}  & M  &\cQ \arrow{l}{q}\arrow[swap]{ul}{\overline{q}}
\end{tikzcd}
\end{eqnarray}

Let us explain the undefined maps:

$\iota_{\mathbf{1}}\colon   \End_{\cQ} M \to M$ is the map that evaluates a $\cQ$-equivariant map $\lambda\colon   M^{\otimes k} \to M$ on the element $\mathbf{1}$, i.e., $\iota_{\mathbf{1}}(\lambda)= \lambda(\mathbf{1},\dots,\mathbf{1}) \in M(k)$.

It is clear that $q$ is a map of right $\cQ$-modules. $\overline{q}\colon   \cQ \to \End_{\cQ} M$ is the map defined in \eqref{end map} taking $N=\cQ$, a right $\cQ$-module over itself and considering the pair $(q,\mu)$.

A left $\cP$-module structure on $M$ is equivalent to a morphism of operads $\cP\to \End M$. The fact that $M$ is a $\cP-\cQ$-bimodule implies that this morphism factors through $\End_{\cQ} M$. This factorizing morphism is what we call $p'$.

It is clear that the left triangle of diagram \eqref{diagram} is commutative. The right triangle is also commutative, since for all $c\in \cQ$ one has 
$$\iota_{\mathbf{1}} ( \overline{q}(c)) = \overline{q}(c)(\mathbf{1},\dots,\mathbf{1}) = q\circ c (\mu(\mathbf{1}),\dots,\mu(\mathbf{1})) = q(c(1_\cQ,\dots,1_\cQ)) = q(c).$$

By Corollary \ref{qi of End operads} $\overline{q}$ is a quasi-isomorphism of operads and by hypothesis $q$ is a quasi-isomorphism, therefore $\iota_{\mathbf{1}}$ is also a quasi isomorphism. Since $p$ is also a quasi-isomorphism we obtain that $p'$ is a quasi-isomorphism of operads.

$M$ is a $\End_{\cQ}M-\cQ$-bimodule and $\cQ$ is naturally a $\cQ-\cQ$-bimodule. The quasi-isomorphism $q\colon   \cQ \to M$  and the quasi-isomorphisms of operads $\overline{q}\colon   \cQ \to \End_{\cQ} M$ and $\id_{\cQ}$ provide us with a quasi-isomorphism between the bimodules $\cQ$ and $M$. 

The only thing that needs to be checked is that the $q$ commutes with the left action, since we already know that it commutes with the right action. 

Let $c\in \cQ(k)$ and $c_1,\dots,c_k\in \cQ$.
$$ \overline{q}(c)(q(c_1),\dots,q(c_k))=q\circ c (\mu \circ q(c_1)\dots, \mu \circ q(c_k))=q(c(c_1,\dots,c_k)).$$

Therefore we have the following zig-zag of quasi-isomorphisms:
$$
\begin{tikzcd}[column sep=0.5em]				
\cP \arrow{d}{p'} &\aol & M \arrow{d}{\id} & \aor & \cQ \arrow{d}{\id} \\
\End_{\cQ} M &\aol & M &\aor & \cQ \\
\cQ \arrow{u}{\overline{q}}& \aol  & \cQ \arrow{u}{q} &\aor & \cQ \arrow{u}{\id}
\end{tikzcd}
$$



We find a second zig-zag connecting the $\cQ-\cQ$-bimodule $\cQ$ to the $\cP-\cP$-bimodule $\cP$ by 

$$
\begin{tikzcd}[column sep=0.5em]				
\cP \arrow{d}{p'} &\aol & \cP \arrow{d}{p'} & \aor & \cP \arrow{d}{p'} \\
\End_{\cQ} M &\aol & \End_{\cQ} M &\aor & \End_{\cQ} M \\
\cQ \arrow{u}{\overline{q}}& \aol  & \cQ \arrow{u}{\overline{q}} &\aor & \cQ \arrow{u}{\overline{q}}
\end{tikzcd}
$$

thus finishing the proof of Theorem \ref{main}, and hence also that of Theorems \ref{thm:main} and \ref{thm:uniqueness}.

\end{proof}

\section{The resolution} \label{sec:resolution}

In this section we construct the resolution $M_\infty\to M$ and show Proposition \ref{prop:resolution}. 
We remark that the existence of a cofibrant resolution is guaranteed by the model structure on the category of $\cP$-$\cQ$ bimodules \cite{Frbook}. However, we will not use the result of op. cit. directly, but give a direct construction in order to keep the exposition elementary and self-contained, and to be able to verify the additional assertions of Proposition \ref{prop:resolution}.
In fact, the bar and cobar constructions for operadic right modules we used have been studied in detail in \cite[section 4]{Frpartition}, and in less generality in \cite{GJ}.

\subsection{Cobar-Bar resolution of operadic modules}\label{subsec:resolution}

Given operads $\cP$, $\cQ$ and a $\cP$-$\cQ$ bimodule $M$ we construct a canonical quasi-free (as right $\cQ$-module) resolution of $M$.

Let us temporarily assume that $\cQ$ is augmented, and denote the kernel of the augmentation by $\bar{\cQ}$.
Let $B(M)$ be the quasi-free right $B(\cQ)$-bicomodule generated by $M$, where $B(\cQ)$ are the bar construction of the operad $\cQ$.

We define $M_\infty$ to be the quasi-free right $\cQ$-module generated by $B(M)$. The elements of $M_\infty$ can be depicted as linear combinations of trees with the top node labeled by an element of $M$, the inner nodes labeled by elements of $\bar{\cQ}[-1]$ and the bottom nodes labeled by elements of $\cQ$. 
Here is an indication of the decoration of an example tree:
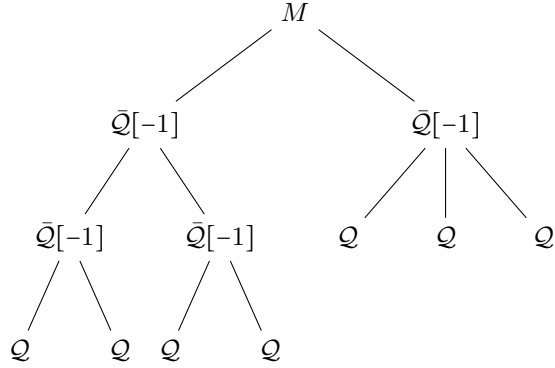
\[
\tikzset{
  treenode/.style = {align=center, text centered,
    font=\sffamily},
  arn_n/.style = {treenode},
}
\begin{tikzpicture}[-,level/.style={sibling distance = 4cm/#1,
  level distance = 1.5cm}] 
\node [arn_n] (f00){$M$}
    child{ node [arn_n] (f20){$\bar\cQ[-1]$} 
            child{ node [arn_n](f21) {$\bar\cQ[-1]$} 
            	child{ node [arn_n] {$\cQ$} edge from parent node[above left]
                         {}} 
							child{ node [arn_n] {$\cQ$}}
            }
            child{ node [arn_n](f10) {$\bar\cQ[-1]$}
							child{ node [arn_n] (f11) {$\cQ$}}
							child{ node [arn_n] (f12){$\cQ$}}
            }                            
    }
    child{ node [arn_n] (f01){$\bar\cQ[-1]$}
            child{ node [arn_n] at ($(f01) +(-1.3,-1.5)$) {$\cQ$} 
            }
           child{ node [arn_n] at ($(f01) +(0,-1.5)$) {$\cQ$}
            }
            child{ node [arn_n] at ($(f01) +(1.3,-1.5)$) {$\cQ$}
            }
		}
; 
\end{tikzpicture}
\]

We define a differential on $M_\infty$ by: 
\begin{itemize}
\item taking the original differential on elements of $M$, $\bar\cQ$ or $\cQ$;
\item contracting an edge connecting the element of $M$ with one element of $\bar{\cQ}$ using the right-$\cQ$ module structure on $M$;
\item contracting an edge connecting two elements of $\bar{\cQ}$ using the operadic composition;
\item contracting the edges connecting one of the lowest nodes labeled by $\bar{\cQ}$ and the respective incoming nodes labeled by elements of $\cQ$ using the operadic composition resulting in a new node labeled by $\cQ$.
\end{itemize}

Here is an indication of how the differential acts by contracting nodes using the right action or the operadic composition.
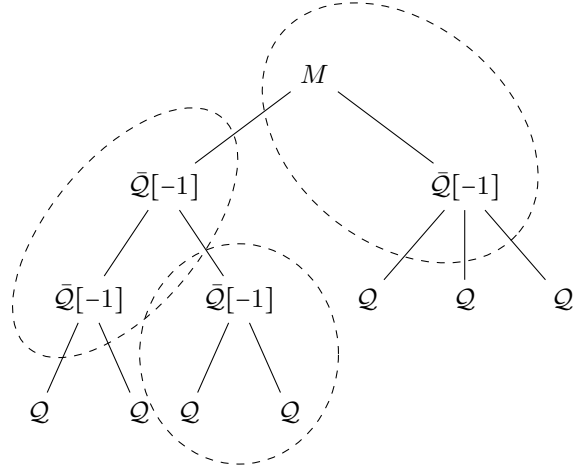
\[
\tikzset{
  treenode/.style = {align=center, text centered,
    font=\sffamily},
  arn_n/.style = {treenode},
}
\begin{tikzpicture}[-,level/.style={sibling distance = 4cm/#1,
  level distance = 1.5cm}] 
\node [arn_n] (f00){$M$}
    child{ node [arn_n] (f20){$\bar\cQ[-1]$} 
            child{ node [arn_n](f21) {$\bar\cQ[-1]$} 
            	child{ node [arn_n] {$\cQ$} edge from parent node[above left]
                         {}} 
							child{ node [arn_n] {$\cQ$}}
            }
            child{ node [arn_n](f10) {$\bar\cQ[-1]$}
							child{ node [arn_n] (f11) {$\cQ$}}
							child{ node [arn_n] (f12){$\cQ$}}
            }                            
    }
    child{ node [arn_n] (f01){$\bar\cQ[-1]$}
            child{ node [arn_n] at ($(f01) +(-1.3,-1.5)$){$\cQ$} 
            }
           child{ node [arn_n] at ($(f01) +(0,-1.5)$) {$\cQ$}
            }
            child{ node [arn_n] at ($(f01) +(1.3,-1.5)$) {$\cQ$}
            }
		}
; 
\node[ellipse, dashed, draw, inner sep=0pt, fit=(f10) (f11) (f12)](FIt1) {};
\node[ellipse,dashed, draw,inner sep=0pt,rotate=-40, fit=(f00) (f01),](FIt2) {};
\draw[dashed,rotate=50] (f20) +(180:.8) ellipse (2 and 1);
\end{tikzpicture}
\]

$M_\infty$ has a natural left $\cP$ action induced from the left $\cP$-module structure one $M$ that is compatible with the right $\cQ$-module structure, making $M_\infty$ a $\cP-\cQ$-module. 

There is a natural projection of $\cP-\cQ$-bimodules $\pi \colon M_\infty \to M$, sending $M\subset B(M) \stackrel{\id}{\to} M$ and all trees with a $\bar \cQ$-labeled node to zero. 

\begin{lemma}\label{lem:resolutionqiso}
The natural projection $M_\infty \to M$ is a quasi-isomorphism.
\end{lemma}
\begin{proof}
The proof is the same as that of \cite[Lemma 6.5.14]{lodayval}, up to minor modifications. 
%
%
%
%
\end{proof}

Next consider the case of a non-augmented operad $\cQ$. We may forget the unit $1_{\cQ}$ and adjoin a new unit $\bbo$ to obtain the augmented operad $\cQ^\bbo$. Concretely, $\cQ^\bbo(1) = \cQ(1) \oplus \K \bbo$, and $\cQ^\bbo(n) = \cQ(n)$ for all $n\neq 1$.
We have  map of unital operads $\cQ^\bbo\to \cQ$ and a map of non-unital operads $\cQ\to \cQ^\bbo$.

We can apply the above construction to the right $\cQ^\bbo$-module $M$ to obtain a resolution of $M$ which we temporarily denote $\widetilde{M}_\infty$.  The right $\cQ^\bbo$-module $\widetilde{M}_\infty$ is quasi-free, $\widetilde{M}_\infty=\Free_{\cQ^\bbo}(B(M))$. Furthermore $\widetilde{M}_\infty$ is a non-free non-unital right $\cQ$ module.
We define $M_\infty$ to be the quasi-free $\cQ$-submodule generated by $B(M)$, i.e., 
\[
M_\infty = \Free_{\cQ}(B(M))\subset \widetilde{M}_\infty.
\]
Note that $M_\infty \subset \widetilde{M}_\infty$  is indeed closed under the differential and that $M_\infty$ is unital, i.e., $1_\cQ$ acts as the identity.
We have the following canonical maps of $\cQ$-modules.
\[
M_\infty \to \widetilde{M}_\infty \to M.
\]
\begin{lemma}
The map $\widetilde{M}_\infty \to M_\infty$ is a quasi-isomorphism.
\end{lemma}

\begin{proof}
Note that via the unital operad map $\cQ^\bbo\to \cQ$ we have a natural map
\[
\widetilde{M}_\infty \to M_\infty
\] 
through which the canonical projection $\widetilde{M}_\infty \to M$ factors.
We also have a canonical map of dg $\bbS$-modules $M\to M_\infty$ sending an element $m\in M(n)$ to the two level tree with root node decorated by $m$ and all leaf nodes decorated by $1_\cQ$. We hence have the following maps of dg $\bbS$-modules:
\[
M\to M_\infty \to \widetilde{M}_\infty \to M_\infty \to M
\] 
inducing maps on homology
\[
H(M)\to H(M_\infty) \to H(\widetilde{M}_\infty) \to H(M_\infty) \to H(M).
\]
The composition of all four maps is the identity on $H(M)$ by construction.
The composition of the last two maps is an isomorphism by Lemma \ref{lem:resolutionqiso}, and hence so must be the composition of the first two. It follows that the second map is surjective and the third injective. But the composition of the second and third maps is the identity on $H(M_\infty)$, and hence the second map must also be injective and the third surjective. Hence all four maps above are isomorphisms.
\end{proof}

\begin{remark}
Note that the resolution $M_\infty$ in the non-augmented case is in fact defined by the same construction as in the augmented case, except that the inner nodes of the trees above are decorated by $\cQ[-1]$ instead of $\bar \cQ[-1]$. 
\end{remark}

\subsection{Lifting property}
The results of this section follow from standard model categorial argument. We will nevertheless spell them out for the sake of completeness.
\begin{lemma}[Lifting property]\label{lem:lifting}
Let $\cQ$ be an operad and $N$ be a quasi-free right $\cQ$-module, $N=\Free_{\cQ}(V)$.
Assume that the generating $\bbS$-module $V$ comes with an exhaustive filtration 
\[
 0=\mF^0V\subset \mF^1V\subset \mF^2V\subset \cdots
\]
such that $d_N \mF^j V\subset \Free_{\cQ}(\mF^{j-1}V) \subset N$. Then $N$ has the left lifting property with respect to surjective quasi-isomorphisms of $\cQ$-modules $f:A\to B$, i.e., given a morphism $g\colon N \to B$, the dashed arrow in diagrams of the following form exists.
\[
 \begin{tikzcd}
  \phantom{A}  & A \arrow[two heads]{d}{f}\\
   N \arrow[dashed]{ur}{s} \arrow{r}{g} & B \\
 \end{tikzcd}
\]

\end{lemma}
\begin{proof}
It is a relative standard cofibrancy proof. 
Given a lifting problem as above, let us construct a lift $s\colon N\to A$, of the map $f$,  commuting with the differential right $\cQ$-module structure.
By assumption $N$ is a quasi-free right $\cQ$-module generated by the $\mathbb S$-module $V$.
We construct $s$ inductively using the filtration on $V$ and at each step we check that for $v\in \mF^k V$ we have $s(d_N v) = d_Fs(v)$ and $p_1\circ s (v) = v$.

For $v\in \mF^1V$ define $s(v)$ to be a (closed) pre-image of $v$.

Let us suppose that $s$ is constructed up until degree $p-1$ and let $v \in \mF^pV$.

Since $d_Nv \in \Free_\cQ (\mF^{p-1} V)$, $s(d_Nv)$ is already defined and it is clearly in the kernel of $d_A$, therefore it represents a homology class $[s(d_Nv)]$.
By induction hypothesis $[f]([s(d_Nv)]) =[d_N g (v)] = [0]$, so, since $f$ is a quasi-isomorphism $s(d_Nv)$ is exact.

$A$ and $N$ can be decomposed as dg vector spaces in $A=H_A \oplus C \oplus C[-1]$ and $B=H_B \oplus D \oplus D[-1]$ in such a way that $C=\im d_A$, $H_A\cong H(A)$, $D=\im d_B$ and $H_B\cong H(B)$ such that under this identification $d_A$ is the  identity map from $C[-1]$ to $C$ and $d_N$ is the identity map from $C[-1]$ to $C$.
Moreover, since $f$ is a surjective quasi-isomorphism, the decomposition can be made in such a way that $f$ restricts to an isomorphism $f\big|_{H_A}\colon   H_A\to H_B$ and $f$ restricts to a surjective map $f\big|_C\colon C\to D$.

We wish to show that we can choose $s(v)$ such that we have simultaneously $s(d_Av)=d_Bs(v)$ and $f\circ s(v)=v$. 

Let us decompose $g(v) = h_B + d_Bd + d' \in H_B \oplus D \oplus D[-1]$ and let us decompose also $s(v) = h_A + d_Ac + c' \in H_A \oplus C \oplus C[-1]$ into the 3 unknown summands that we wish to find.

Since  $s(d_Nv)$ is exact, let us define $c'$ such that $s(d_Nv) = d_Ac'$. This not only guarantees that $s(d_Nv)=d_As(v)$ but also tells us that the projections of $d'$ and $f(c')$ to $D[-1]$ are the same, by taking $f$ on both sides of the equality.

 To solve $f\circ s(v)=v$, we notice that this is equivalent to 
$$f(h_A) + f(d_Ac) = h_B + d_Bd + (d'-f(c')),$$
but since the right-hand side of the equation is in $H_B \oplus C$ and both ${f}\big|_{H_A} $ and $f\big|_{D}$ are surjective, $h_A$ and $d_Ac$ can be chosen such that the equality holds.

Continuing the induction we obtain the desired lift $s$.
\end{proof}

The following result is an (almost) immediate Corollary of the previous Lemma, together with the standard ``surjective trick''.
\begin{lemma}\label{iffree}
Let $\cQ$ and $N$ be as in  be as in Lemma \ref{lem:lifting}. Let $f\colon \cQ\to N$ be a quasi-isomorphism of right $\cQ$-modules such that $f(1_{\cQ})\in V$. Then there exists a map of right dg $\cQ$-modules $g\colon N \to \cQ$ such that $g\circ f =\id_\cQ$.
\end{lemma}
The proof will show that $f$ is a cofibration, from which the statement easily follows.
\begin{proof}
Consider the right $\cQ$-module $F = \cQ \oplus N \oplus N[1]$ with differential $d_F$ equal to the sum of the given differentials plus one extra piece that acts as the identity from $N$ to $N[1]$.
It comes with a surjective quasi-isomorphism $p_1\colon F\twoheadrightarrow N$ sending $(q,n,n')\in \cQ\oplus N \oplus N[1]$ to $f(q)+n$. Furthermore, one has the natural projection (quasi-isomorphism of $\cQ$-modules) $p_2\colon F\to \cQ$.

We apply Lemma \ref{lem:lifting} to the lifting problem 
\[
 \begin{tikzcd}
  \phantom{A}  & F \arrow[two heads]{d}{p_1}\\
   N \arrow[dashed]{ur}{g'} \arrow{r}{=} & N \\
 \end{tikzcd}
\]
to construct a map $g':N\to F$. In fact, looking at the proof of Lemma \ref{lem:lifting} we may choose $g'$ such that $g'(f(1_{\cQ}))=1_{\cQ}\in \cQ \subset F$.
The composition $g:=p_2\circ g'$ satisfies 
\[
 g(f(1_{\cQ}))=p_2(g'(f(1_{\cQ}))=1_{\cQ}
\]
and hence $g\circ f =\id_\cQ$ since $\cQ$ is generated by $1_{\cQ}$ as right $\cQ$-module.
\end{proof}

\begin{remark}\label{rem:Minftycofibrant}
We note that the right $\cQ$-module $N=M_\infty$ satisfies the conditions of Lemma \ref{lem:lifting}.
 Indeed, we may define the required filtration on the generating $\bbS$-module $B(M)$ as follows.
 Let $d_{int}$ be the piece of the differential on $M_\infty$ that stems from the internal differentials on $\cM$ and $\cQ$, then:
 \begin{align*}
  \mF^{2p-1}B(M) &= \{\text{span of trees with $<p$ inner nodes}\} \oplus \\ &\oplus \{\text{linear combinations of trees with $=p$ inner nodes closed under $d_{int}$}\} \\
  \mF^{2p}B(M) &= \{\text{span of trees with $\leq p$ inner nodes}\} \\
 \end{align*}
\end{remark}

\subsection{Proof of Proposition \ref{prop:resolution}}

We first claim that $M_\infty$ is an operadic $\cP$-$\cQ$-torsor. Indeed, denote the element $\mathbf 1\in M(1)$, whose existence is guaranteed by Definition \ref{def:optorsor}, temporarily by $\mathbf 1_M$. Since $M_\infty\to M$ is a quasi-isomorphism we may pick some degree zero cycle $\mathbf 1_{M_\infty}\in M$ lifting $\mathbf 1_M$. We need to check that the induced maps \eqref{equ:modmaps} (for $M$ replaced by $M_\infty$ and $\mathbf 1$ by $\mathbf 1_{M_\infty}$) are quasi-isomorphisms. But in the following commutative diagram of $\bbS$-modules
\[
 \begin{tikzcd}
  \cP \arrow{r}{} \arrow{d}{\id}& M_\infty \arrow{d}{}& \cQ \arrow{l}{} \arrow{d}{\id}\\
  \cP \arrow{r}{} & M & \cQ \arrow{l}{}\\
 \end{tikzcd}
\]
all arrows except the upper horizontal are known to be quasi-isomorphisms, and hence so have to be the upper horizontal arrows.

The existence of the map $\mu$ in Proposition \ref{prop:resolution} is a direct consequence of Lemma \ref{iffree} together with Remark \ref{rem:Minftycofibrant}.
\hfill \qed

\section{Application: Deligne's Conjecture} \label{sec:deligne}

As a consequence of Theorem \ref{thm:main} we may present a very short and elegant proof of the following conjecture due to P. Deligne \cite{Deligne}.

\begin{theorem}[Deligne's Conjecture]
For any associative algebra $A$, there exists a natural action of the chains operad of the Little Disks Operad on the Hochschild complex $C^\bullet (A,A)$, such that the induced action of the Homology of the Little Disks Operad on the Hochschild Homology of $A$ corresponds to its standard Gerstenhaber algebra structure.
\end{theorem}
It is well known that the  Hochschild complex $C^\bullet (A,A)$ carries a natural action of the braces operad $\Br$, see \cite{Baetal} for details.
Furthermore, we may take as a model of the little disks operad the operad of (semi-algebraic) chains $C(\FM_2)$ of the Fulton-MacPherson operad $\FM_2$, cf. \cite{KS1, GJ, HLTV}.
Hence to show the Deligne conjecture it suffices to produce a (suitable) map $C(\FM_2)\to \Br$.
Let $C_{m,n}$ be the configuration space of $m$ points in the upper half-plane and $n$ points on the boundary, suitable compactified so that the spaces $\FM_2(n)$ and $C_{m,n}$ together assemble to form a colored operad, modelling the Swiss Cheese operad.
The following result has been shown in \cite{brformality}.
\begin{proposition}[Proposition 2 of \cite{brformality}]\label{prop:bimodstruct}
The semi-algebraic chains $C(C_{\bullet,0})$ carry a natural $\Br-C(\FM_2)$-bimodule structure, such that the induced $H(\Br)-H(C(\FM_2))$(i.e., the $\e_2$-$\e_2$-)bimodule structure on $H(C_{\bullet,0})\cong \e_2$ is the canonical one.
\end{proposition}

\begin{proof}[Proof of Deligne's conjecture]
Proposition \ref{prop:bimodstruct} states in particular that $C(C_{\bullet,0})$ is an operadic $\Br-C(\FM_2)$-quasi-torsor.
In view of Theorem \ref{thm:main} the Deligne conjecture follows.
%
\end{proof}

\begin{remark}
The Deligne conjecture has seen various proofs in recent years, including those by McClure and Smith \cite{M-Smith}, Kontsevich and Soibelman \cite{KS1}, Tamarkin \cite{Dima-another} and others. The above proof is however the shortest and most natural we know, even if one takes into account the definition of the bimodule structure of Proposition \ref{prop:bimodstruct} in \cite{brformality}.
\end{remark}

\section{Application: Homotopy Braces Formality Theorem} \label{sec:brformality}
M. Kontsevich's formality Theorem is deformation quantization \cite{Ko} states that there is an $L_\infty$ quasi-isomorphism between the multi vector fields and multi differential operators on a smooth manifold
\[
T_{\rm poly}[1]\to D_{\rm poly}[1].
\]
On the right hand side there is furthermore a natural action of the braces operad. The homotopy braces formality Theorem proven by the second author \cite{brformality} states that M. Kontsevich's $L_\infty$-morphism may be extended to a homotopy braces morphism. 
One of the main steps in the proof is the construction of a quasi-isomorphism of operads and bimodules (cf. Proposition \ref{prop:bimodstruct})
\begin{equation}\label{equ:tobefound}
\begin{tikzcd}[column sep=0.5em]				
\Br_\infty \arrow{d}{} &\aol &\Br_\infty^{\text{bimod}} \arrow{d}{} & \aor &\Br_\infty \arrow{d}{} \\
\Br & \aol  & C(C_{\bullet,0})  &\aor & C(\FM_2) 
\end{tikzcd}
\end{equation}
where $\Br_\infty$ is a cofibrant resolution of the braces operad $\Br$ and $\Br_\infty^{\text{bimod}}$ is a cofibrant resolution of the $\Br$-$\Br$ operadic bimodule $\Br$.

This construction takes a significant amount of space and effort in \cite{brformality}. Given Theorem \ref{thm:uniqueness} it can be cut short to a few lines. First, since $C(C_{\bullet,0})$ is an operadic torsor by Proposition \ref{prop:bimodstruct}, we may use Theorem \ref{thm:uniqueness} to connect it by a zig-zag of quasi-ismorphisms of operads and bimodules to $\Br\aol\Br\aor\Br$. But one can continue the zig-zag as follows
$$
\begin{tikzcd}[column sep=0.5em]
\Br_\infty \arrow{d}{} &\aol & \Br_\infty^{\text{bimod}}  \arrow{d}{} & \aor &\Br_\infty  \arrow{d}{} \\
\Br \arrow{d}{} & \aol& \Br\arrow{d}{}& \aor& \Br\arrow{d}{} \\
\cdots & \aol& \cdots & \aor& \cdots \\
\Br \arrow{u}{}& \aol  & C(C_{\bullet,0}) \arrow{u}{}  &\aor & C(\FM_2) \arrow{u}{} .
\end{tikzcd}
$$
The upper line is quasi-free and by lifting (up to homotopy) across the zig-zag we may find the desired quasi-isomorphism \eqref{equ:tobefound}. 

More details can be found in the thesis of the first author \cite{RicardoThesis}, where the above trick is used to generalize the result of \cite{brformality} from the Hochschild to the cyclic setting.

\end{document}